\documentclass[11pt]{article}
\usepackage{amssymb}
\usepackage{amsmath}
\usepackage{amsthm}
\usepackage{graphicx}
\usepackage{enumerate}
\usepackage[all,2cell]{xy} \UseAllTwocells \SilentMatrices
\reversemarginpar
\numberwithin{equation}{section}
\theoremstyle{definition}

\DeclareMathOperator{\Sp}{Sp}
\DeclareMathOperator{\Mp}{Mp}

\begin{document}
\newcommand{\cb}[1]{\left\{#1\right\}}
\newcommand{\lb}[1]{\left(#1\right)}
\newcommand{\ls}[1]{\left[#1\right]}
\newcommand{\lr}[1]{\left\langle#1\right\rangle}
\newcommand{\lat}[2]{\left.#1\right|_{#2}}
\newcommand{\dd}[2]{\frac{d{#1}}{d{#2}}}
\newcommand{\pp}[2]{\frac{\partial{#1}}{\partial{#2}}}
\newcommand{\tp}{\ensuremath{\widetilde{\phi}}}
\newcommand{\hp}{\ensuremath{\widehat{\phi}}}
\newcommand{\tr}{\ensuremath{\widetilde{\rho}}}
\newcommand{\hr}{\ensuremath{\widehat{\rho}}}
\newcommand{\hxi}{\ensuremath{\widehat{\xi}}}
\newcommand{\hx}{\ensuremath{\widehat{x}}}
\newcommand{\hxs}{\ensuremath{\widehat{x}^*}}
\newcommand{\txi}{\ensuremath{\widetilde{\xi}}}
\newcommand{\hy}{\ensuremath{\widehat{y}}}
\newcommand{\hys}{\ensuremath{\widehat{y}^*}}
\newcommand{\hz}{\ensuremath{\widehat{z}}}
\newcommand{\hzs}{\ensuremath{\widehat{z}^*}}
\newcommand{\sC}{\ensuremath{\mathcal{C}}}
\newcommand{\C}{\ensuremath{\mathbb{C}}}
\newcommand{\N}{\ensuremath{\mathbb{N}}}
\newcommand{\Q}{\ensuremath{\mathbb{Q}}}
\newcommand{\R}{\ensuremath{\mathbb{R}}}
\newcommand{\Z}{\ensuremath{\mathbb{Z}}}
\newcommand{\contr}{\ensuremath{\lrcorner\,}}
\newcommand{\Det}{\ensuremath{\mbox{Det}}}
\newcommand{\maps}[1]{\ensuremath{\stackrel{#1}{\longrightarrow}}}
\newcommand{\two}[2]{\ensuremath{\lb{\begin{array}{c}{#1}\\{#2}\end{array}}}}
\newcommand{\four}[4]{\ensuremath{\lb{\begin{array}{cc}{#1}&{#2}\\{#3}&{#4}\end{array}}}}

\newtheorem{theorem}{Theorem}[section]
\newtheorem{example}[theorem]{Example}
\newtheorem{definition}[theorem]{Definition}
\newtheorem{lemma}[theorem]{Lemma}
\newtheorem{proposition}[theorem]{Proposition}

\title{Dynamical Invariance of a New Metaplectic-c Quantization Condition}
\author{Jennifer Vaughan\\University of Toronto\\ \texttt{jennifer.vaughan@mail.utoronto.ca}}
\date{\today}

\maketitle

\begin{abstract}
Metaplectic-c quantization was developed by Robinson and Rawnsley as an alternative to the classical Kostant-Souriau quantization procedure with half-form correction.  Given a metaplectic-c quantizable symplectic manifold $(M,\omega)$ and a smooth function $H:M\rightarrow\R$, this paper proposes a condition under which $E$, a regular value of $H$, is a quantized energy level for the system $(M,\omega,H)$.  The condition is evaluated on a bundle over $H^{-1}(E)$.  Under modest assumptions, if $E$ is a quantized energy level of $(M,\omega,H)$, then the symplectic reduction of $(M,\omega)$ at $E$ admits a metaplectic-c prequantization, provided the reduction is a smooth manifold.  Moreover, the quantization condition is dynamically invariant:  if there are two functions $H_1$, $H_2$ on $M$ such that $H_1^{-1}(E_1)=H_2^{-1}(E_2)$ for regular values $E_1$, $E_2$, then $E_1$ is a quantized energy level for $(M,\omega,H_1)$ if and only if $E_2$ is a quantized energy level for $(M,\omega,H_2)$.
\end{abstract}

\section{Introduction}

The Kostant-Souriau quantization recipe with half-form correction applies to a symplectic manifold that admits a prequantization circle bundle and a metaplectic structure.  Robinson and Rawnsley \cite{rr1} gave an alternative to this process in which the circle bundle and metaplectic structure are replaced by a single object called a metaplectic-c prequantization.  Metaplectic-c quantization applies to a broader class of manifolds than Kostant-Souriau quantization.

It is well known from physics that the quantum mechanical versions of systems such as the harmonic oscillator and the hydrogen atom are restricted to discrete energy levels.  More generally, suppose $(M,\omega)$ is a quantizable symplectic manifold and $H:M\rightarrow\R$ is a function on $M$.  If we think of $H$ as a Hamiltonian energy function, then we can ask what it means for a regular value $E$ of $H$ to be a quantized energy level for the system $(M,\omega,H)$.

One possible answer to this question involves the construction of the symplectic reduction.  The orbits of the Hamiltonian vector field $\xi_H$ partition the level set $H^{-1}(E)$.  If $M_E$, the space of orbits, is a manifold, then $\omega$ induces a symplectic form $\omega_E$ on $M_E$, and $(M_E,\omega_E)$ is the symplectic reduction of $(M,\omega)$ at $E$.  The values of $E$ for which this new symplectic manifold is quantizable can be taken to be the quantized energy levels of the system.  This definition has been applied to the hydrogen atom in the context of Kostant-Souriau quantization \cite{sim3}, and to the harmonic oscillator in the context of metaplectic-c quantization \cite{rr1}.  In both cases, the physically predicted energy spectrum is obtained.  However, this definition is limited to cases where the space of orbits is a manifold.  It would be more convenient to have a quantized energy condition that is evaluated over the original manifold, rather than the quotient.

In this paper, we propose a definition for a quantized energy level of $(M,\omega,H)$ in the case where $(M,\omega)$ admits a metaplectic-c prequantization.  Given one additional assumption, which we describe in Section \ref{subsec:mpcred}, if $E$ is a quantized energy level of $(M,\omega,H)$, then the symplectic reduction $(M_E,\omega_E)$ is metaplectic-c quantizable, provided that it is a manifold.  The quantized energy condition is evaluated on a bundle over the level set $H^{-1}(E)$.  

This scenario was first studied by Robinson \cite{rob1}, and our work is strongly motivated by his results.  However, we choose a more flexible condition than the one Robinson considered.  As we will show, if $H_1$ and $H_2$ are two functions on $(M,\omega)$ with regular values $E_1$ and $E_2$, respectively, such that $H_1^{-1}(E_1)=H_2^{-1}(E_2)$, then under our definition, $E_1$ is a quantized energy level of $(M,\omega,H_1)$ if and only if $E_2$ is a quantized energy level of $(M,\omega,H_2)$.  In other words, the quantization condition only depends on geometry of the level set, and not on the dynamics of a particular Hamiltonian.  This is not true of the definition considered in \cite{rob1}.

In Section \ref{sec:mpcgp}, we review the definitions of the metaplectic-c group and metaplectic-c prequantization.  Section \ref{sec:energy} summarizes Robinson's results concerning symplectic reduction, and concludes with our proposed definition for a quantized energy level of a metaplectic-c quantizable system.  The proof of the dynamical invariance of our definition is given in Section \ref{sec:dynam}.  

In Section \ref{sec:sho}, we consider the example of the harmonic oscillator.  This example serves two purposes.  First, we demonstrate a computational technique in which a local change of coordinates on $M$ is lifted to the metaplectic-c prequantization.  Second, we state an example of a function on $M$ that has exactly one level set in common with the harmonic oscillator, and show that the quantization condition on that level set is identical for both functions.  Finally, in Section \ref{sec:ks}, we adapt the quantized energy definition to apply to a Kostant-Souriau quantizable symplectic manifold.  We show that the Kostant-Souriau and metaplectic-c definitions have almost identical properties, but the metaplectic-c version is the only one that yields the correct quantized energy levels for the harmonic oscillator.

\section{The Metaplectic-c Group and Metaplectic-c Prequantization}\label{sec:mpcgp}

The material in this section is a summary of results originally presented by Robinson and Rawnsley \cite{rr1}.  More detail, including proofs of the properties that we state, can be found in \cite{rr1}.

\subsection{Definitions on Vector Spaces}\label{subsec:vtspc}

Let $V$ be an $n$-dimensional complex vector space with Hermitian inner product $\lr{\cdot,\cdot}$.  If we view $V$ as a $2n$-dimensional real vector space, then the action of the scalar $i\in\C$ becomes the real automorphism $J:V\rightarrow V$.  Define the real bilinear form $\Omega$ on $V$ by $$\Omega(v,w)=\mbox{Im}\lr{v,w},\ \ \forall v,w\in V.$$  Then $(V,\Omega)$ is a $2n$-dimensional symplectic vector space.  The Hermitian and symplectic structures are compatible in the sense that $$\lr{v,w}=\Omega(Jv,w)+i\Omega(v,w),\ \ \forall v,w\in V.$$

The symplectic group $\Sp(V)$ is the group of real automorphisms $g:V\rightarrow V$ such that $\Omega(gv,gw)=\Omega(v,w)$ for all $v,w\in V$.  Its connected double cover is the metaplectic group $\Mp(V)$.  Let $\Mp(V)\maps{\sigma}\Sp(V)$ denote the covering map, and note that $\sigma^{-1}(I)$ is a subgroup of $\Mp(V)$ that is isomorphic to $\Z_2$.   This subgroup is used to construct the metaplectic-c group $\Mp^c(V)$:  $$\Mp^c(V)=\Mp(V)\times_{\Z_2}U(1).$$

The \emph{projection map} on $\Mp^c(V)$, which we also denote by $\sigma$, is the group homomorphism such that $$1\rightarrow U(1)\rightarrow\Mp^c(V)\maps{\sigma}\Sp(V)\rightarrow 1$$ is a short exact sequence, and such that the restriction of $\sigma$ to $\Mp(V)\subset\Mp^c(V)$ is exactly the double covering map $\Mp(V)\maps{\sigma}\Sp(V)$.  The \emph{determinant map} $\eta$ is the group homomorphism on $\Mp^c(V)$ such that $$1\rightarrow\Mp(V)\rightarrow\Mp^c(V)\maps{\eta}U(1)\rightarrow 1$$ is a short exact sequence, and such that the restriction of $\eta$ to $U(1)\subset\Mp^c(V)$ is given by $\eta(\lambda)=\lambda^2$ for all $\lambda\in U(1)$.  At the level of Lie algebras, the map $\sigma_*\oplus\frac{1}{2}\eta_*$ yields the identification $$\mathfrak{mp}^c(V)=\mathfrak{sp}(V)\oplus\mathfrak{u}(1).$$

Given $g\in\Sp(V)$, let $$C_g=\frac{1}{2}(g-JgJ).$$  By construction, $C_g$ commutes with $J$, so it is a complex linear transformation of $V$.  The determinant of $C_g$ as a complex transformation is written $\mbox{Det}_\C C_g$.  It can be shown that $C_g$ is always invertible, so $\mbox{Det}_\C C_g$ is a nonzero complex number.

A useful embedding of $\Mp^c(V)$ into $\Sp(V)\times\C\setminus\cb{0}$ can be defined as follows.  An element $a\in\Mp^c(V)$ is mapped to the pair $(g,\mu)\in\Sp(V)\times\C\setminus\cb{0}$, where $g\in\Sp(V)$ satisfies $\sigma(a)=g$, and $\mu\in U(1)$ satisfies $\eta(a)=\mu^2\Det_\C C_g$.  The latter condition defines $\mu$ up to a sign, and we adopt the convention that the parameters of $I\in\Mp^c(V)$ are $(I,1)$.  This choice uniquely determines the sign of $\mu$ for all $a\in\Mp^c(V)$.  Following \cite{rr1}, we refer to $(g,\mu)$ as the \emph{parameters} of $a$.

We note two properties of this parametrization, both of which will be used in Section \ref{sec:sho}.  First, if $a\in\Mp(V)\subset\Mp^c(V)$, then $\eta(a)=1$, and so the parameters of $a$ are $(g,\mu)$ where $\sigma(a)=g$ and $\mu^2\Det_\C C_g=1$.  Second, if $\lambda\in U(1)\subset\Mp^c(V)$, then the parameters of $\lambda$ are $(I,\lambda)$, and for any other element $a\in\Mp^c(V)$ with parameters $(g,\mu)$, the parameters of the product $a\lambda$ are $(g,\mu\lambda)$.

\subsection{Vector Subspaces and Quotients}\label{subsec:vtsubspc}

Let $W\subset V$ be a real subspace of codimension 1.  Define the symplectic orthogonal of $W$ to be $$W^\perp=\cb{v\in V:\Omega(v,W)=0}.$$  Then $W^\perp$ is a one-dimensional subspace of $W$.  The form $\Omega$ induces a symplectic form $\Omega_W$ on the quotient space $W/W^\perp$.  We write an element of $W/W^\perp$ as an equivalence class $[w]$ for some $w\in W$.  

Let $\Sp(V;W)$ be the subgroup of $\Sp(V)$ consisting of those symplectic automorphisms that preserve $W$:  $$\Sp(V;W)=\cb{g\in\Sp(V):gW=W}.$$  Then elements of $\Sp(V;W)$ also preserve $W^\perp$, and there is a group homomorphism $$\Sp(V;W)\maps{\nu}\Sp(W/W^\perp),$$ defined by $$(\nu g)[w]=[gw],\ \ \forall g\in\Sp(V;W), \forall w\in W.$$

Let $\Mp^c(V;W)\subset\Mp^c(V)$ be the preimage of $\Sp(V;W)$ under the projection map $\sigma$.  Robinson and Rawnsley \cite{rr1} constructed a group homomorphism $\Mp^c(V;W)\maps{\hat{\nu}}\Mp^c(W/W^\perp)$ such that the diagram below commutes.
\begin{center}
$\xymatrix{
\Mp^c(V)\supset\Mp^c(V;W) \ar[r]^(0.6){\hat{\nu}} \ar[d]^\sigma & \Mp^c(W/W^\perp) \ar[d]^\sigma \\
\Sp(V)\supset\Sp(V;W) \ar[r]^(0.6){\nu} & \Sp(W/W^\perp)
}$
\end{center}
The map $\hat{\nu}$ has the following property.  Given $a\in\Mp^c(V;W)$, let $\chi(a)=\mbox{Det}_\R(\sigma(a)|W^\perp)$.  Then $\chi$ is a real-valued character on $\Mp^c(V;W)$, and it satisfies $$\eta\circ\hat{\nu}(a)=\eta(a)\mbox{sign}\,\chi(a),\ \ \forall a\in\Mp^c(V;W).$$   Since $\mbox{sign}\,\chi(a)$ is identically $1$ on a neighborhood of $I$, we see that $\eta_*\circ\hat{\nu}_*=\eta_*$ as maps between Lie algebras.

\subsection{Definitions Over a Symplectic Manifold}\label{subsec:mpcquant}

Let $(M,\omega)$ be a $2n$-dimensional symplectic manifold, and assume that a $2n$-dimensional model symplectic vector space $(V,\Omega)$ with complex structure $J$ has been fixed.  The groups $\Sp(V)$ and $\Mp^c(V)$ are the structure groups for the bundles that we now define.  

\begin{definition}
The \textbf{symplectic frame bundle} $\Sp(M,\omega)\maps{\rho}M$ is a right principal $\Sp(V)$ bundle over $M$ given by $$\Sp(M,\omega)_m=\cb{b:V\rightarrow T_mM:b\mbox{ is a symplectic linear isomorphism}},\ \ \forall m\in M.$$  The group $\Sp(V)$ acts on the fibers by precomposition.
\end{definition}

\begin{definition}
A \textbf{metaplectic-c structure} for $(M,\omega)$ consists of a right principal $\Mp^c(V)$ bundle $P\maps{\Pi}M$ and a map $P\maps{\Sigma}\Sp(M,\omega)$ such that $$\Sigma(q\cdot a)=\Sigma(q)\cdot\sigma(a),\ \ \forall q\in P,\ \forall a\in\Mp^c(V),$$ and such that the following diagram commutes.
$$\xymatrix{
P \ar[r]^(0.4){\Sigma} \ar[d]^\Pi & \Sp(M,\omega) \ar[dl]^\rho \\
(M,\omega)
}$$
\end{definition}

\begin{definition}
A \textbf{metaplectic-c prequantization} of $(M,\omega)$ is a metaplectic-c structure $(P,\Sigma)$ for $(M,\omega)$, together with a $\mathfrak{u}(1)$-valued one-form $\gamma$ on $P$ such that:
\begin{enumerate}[(1)]
\item $\gamma$ is invariant under the principal $\Mp^c(V)$ action;

\item for any $\alpha\in\mathfrak{mp}^c(V)$, if $\partial_\alpha$ is the vector field on $P$ generated by the infinitesimal action of $\alpha$, then $\gamma(\partial_\alpha)=\frac{1}{2}\eta_*\alpha$;

\item $d\gamma=\frac{1}{i\hbar}\Pi^*\omega$.
\end{enumerate}
\end{definition}

A similar structure, called a spin-c prequantization, was studied in \cite{ckt,f1,ggk,gk}.

\section{Quantized Energy Levels}\label{sec:energy}

Given a symplectic manifold $(M,\omega)$ and a smooth function $H:M\rightarrow\R$, the symplectic reduction of $(M,\omega)$ at the regular value $E$ of $H$ is constructed by taking the quotient of the level set $H^{-1}(E)$ by the orbits of the Hamiltonian vector field $\xi_H$.  Suppose $(M,\omega)$ admits a metaplectic-c prequantization $(P,\Sigma,\gamma)$.  There is a natural lift of $\xi_H$ to a vector field $\txi_{H}$ on $\Sp(M,\omega)$; $\txi_H$ can then be lifted to a vector field $\hxi_H$ on $P$ that is horizontal with respect to $\gamma$.  Robinson \cite{rob1} constructed a certain associated bundle $P_S\rightarrow H^{-1}(E)$, on which $\hxi_H$ induces a vector field, and he examined the conditions under which the quotient of $P_S$ by the induced orbits of $\hxi_H$ yields a metaplectic-c prequantization for the symplectic reduction of $(M,\omega)$ at $E$.

In this section, we review the key results from \cite{rob1}, then state our alternative definition for a quantized energy level of the system $(M,\omega,H)$.  Our definition is evaluated on the bundle $P_S$, and ensures that the symplectic reduction acquires a metaplectic-c prequantization under the same conditions as Robinson's definition.  Further, as we will show in Section 4, our definition depends only on the geometry of the level set $H^{-1}(E)$, and not on the specific choice of $H$.  This is an improvement over the definition from \cite{rob1}.

\subsection{Lifting the Hamiltonian Vector Field}\label{subsec:hamvt}

Assume that the model vector space $V$ has been chosen as in Section \ref{subsec:vtspc}.  Let $(M,\omega)$ be a symplectic manifold that admits a metaplectic-c prequantization $(P,\Sigma,\gamma)$.  The definitions from Section \ref{subsec:mpcquant} yield the three-level structure $$(P,\gamma)\maps{\Sigma}\Sp(M,\omega)\maps{\rho}(M,\omega),\ \ \rho\circ\Sigma=\Pi.$$  

Let $H:M\rightarrow\R$ be a smooth function on $M$.  We define the Hamiltonian vector field $\xi_H$ using the convention that $$\xi_H\contr\omega=dH.$$  Let the flow of $\xi_H$ be $\phi^t$, and note that $\phi^t$ is a symplectomorphism from its domain to its range for each $t$.  

The following procedure for lifting $\xi_H$ and $\phi^t$ was laid out in \cite{rr1}.  Since $\phi^t$ is a symplectomorphism, the pushforward $\phi^t_*:T_mM\rightarrow T_{\phi^tm}M$ is a symplectic isomorphism for each $m\in M$.  Let $\tp^t:\Sp(M,\omega)\rightarrow\Sp(M,\omega)$ be defined by $$\tp^t(b)=\phi^t_*\circ b,\ \ \forall b\in\Sp(M,\omega).$$  Then $\tp^t$ is a flow on $\Sp(M,\omega)$.  Let $\txi_H$ be the vector field on $\Sp(M,\omega)$ whose flow is $\tp^t$.  Define
the vector field $\hxi_H$ on $P$ to be the lift of $\txi_H$ that is horizontal with respect to the one-form $\gamma$, and let the flow of $\hxi_H$ be $\hp^t$.

\subsection{Quotients and Reduction}\label{subsec:mpcred}

We will now restrict our attention to a particular level set.  Let $E\in\R$ be a regular value of $H$, and let $S=H^{-1}(E)$.  The null foliation $TS^\perp$ is defined fiberwise over $S$ by $$T_sS^\perp=\cb{\zeta\in T_sM:\omega(\zeta,T_sS)=0},\ \ \forall s\in S.$$  Then $T_sS^\perp$ is a one-dimensional subspace of the tangent space $T_sS$, and $T_sS^\perp=\mbox{span}\cb{\xi_H(s)}$.  

The constructions that follow are taken from \cite{rob1}.  As in Section \ref{subsec:vtsubspc}, fix a subspace $W\subset V$ of codimension $1$.  By definition, an element $b\in\Sp(M,\omega)_m$ is a symplectic linear isomorphism $b:V\rightarrow T_mM$.  Let $\Sp(M,\omega;S)$ be the subset of $\Sp(M,\omega)|_{S}$ given by $$\Sp(M,\omega;S)_s=\cb{b\in\Sp(M,\omega)_s:bW=T_sS},\ \ \forall s\in S.$$  Then $\Sp(M,\omega;S)$ is a principal $\Sp(V;W)$ bundle over $S$.  Note that any $b\in\Sp(M,\omega;S)_s$ maps $W^\perp$ to $T_sS^\perp$.

Viewing $P$ as a circle bundle over $\Sp(M,\omega)$, let $P^S$ be the result of restricting $P$ to $\Sp(M,\omega;S)$.  Then $P^S$ is a principal $\Mp^c(V;W)$ bundle over $S$.   Let $\gamma^S$ be the pullback of $\gamma$ to $P^S$.  After all of these restrictions, we have constructed the new three-level structure $$(P^S,\gamma^S)\rightarrow\Sp(M,\omega;S)\rightarrow S.$$

For each $s\in S$, $T_sS/T_sS^\perp$ is a symplectic vector space with symplectic structure induced by $\omega_s$.  Let $\Sp(TS/TS^\perp)$ be the symplectic frame bundle for $TS/TS^\perp$, modeled on $(W/W^\perp,\Omega_W)$:  $$\Sp(TS/TS^\perp)_s=\cb{b':W/W^\perp\rightarrow T_sS/T_sS^\perp:b'\ \mbox{is a symplectic linear isomorphism}},\ \ \forall s\in S.$$  Recall the group homomorphism $\nu:\Sp(V;W)\rightarrow\Sp(W/W^\perp)$ from Section \ref{subsec:vtsubspc}.  The bundle associated to $\Sp(M,\omega;S)$ by the map $\nu$ can be naturally identified with $\Sp(TS/TS^\perp)$.

Next, let $P_S\rightarrow S$ be the bundle associated to $P^S\rightarrow S$ by the group homomorphism $\hat{\nu}:\Mp^c(V;W)\rightarrow\Mp^c(W/W^\perp)$.  Then $P_S$ is a principal $\Mp^c(W/W^\perp)$ bundle over $S$, and a circle bundle over $\Sp(TS/TS^\perp)$.  The one-form $\gamma^S$ induces a connection one-form $\gamma_S$ on $P_S$.  This completes the construction of another three-level structure, $$(P_S,\gamma_S)\rightarrow\Sp(TS/TS^\perp)\rightarrow S.$$

Now let us consider the actions of $\phi^t$, $\tp^t$ and $\hp^t$ on all of these bundles.  First, it is clear that $\phi^t$ preserves $S$, since $\xi_H(s)\in T_sS$ at each $s\in S$.  Therefore $\tp^t$ restricts to a flow on $\Sp(M,\omega;S)$, and so $\hp^t$ restricts to a flow on $P^S$.  One can verify that $\tp^t$ and $\hp^t$ are equivariant with respect to the principal $\Sp(V)$ and $\Mp^c(V)$ actions, respectively, which implies that $\tp^t$ induces a flow on the associated bundle $\Sp(TS/TS^\perp)$, and $\hp^t$ induces a flow on $P_S$.  We let $\tp^t$ and $\hp^t$ also denote the respective flows induced on $\Sp(TS/TS^\perp)$ and $P_S$.

Suppose the null foliation is fibrating, so that the quotient of $S$ by the orbits of $\xi_H$ is a manifold.  Let the quotient manifold be $M_E$, and let $\omega_E$ be the symplectic form on $M_E$ induced by $\omega$.  As discussed in \cite{rob1}, if the quotient of $\Sp(TS/TS^\perp)$ by $\tp^t$ is well defined, then it is naturally isomorphic to the symplectic frame bundle $\Sp(M_E,\omega_E)$.  To ensure that this quotient is well defined, it is sufficient to require that if $s\in S$ is fixed by $\phi^t$, then $\phi^t_*$ is the identity on $T_sS$.  The culmination of all of these steps is shown below.

\begin{center}
$\xymatrix{
(P,\gamma) \ar[d]^{\Sigma} & \ar[l]_{\mbox{\footnotesize incl.}} (P^S,\gamma^S) \ar[d] \ar[r]^{\hat{\nu}} & (P_S,\gamma_S) \ar[d] & \\
\Sp(M,\omega) \ar[d]^{\rho} & \ar[l]_{\mbox{\footnotesize incl.}} \Sp(M,\omega;S) \ar[d] \ar[r]^{\nu} & \Sp(TS/TS^\perp) \ar[d] \ar[r]^{/\tp^t} & \Sp(M_E,\omega_E) \ar[d] \\
(M,\omega) & \ar[l]_{\mbox{\footnotesize incl.}} S \ar[r]^{=} & S \ar[r]^{/\phi^t} & (M_E,\omega_E)
}$
\end{center}

The obvious way to complete the picture is to factor $(P_S,\gamma_S)$ by $\hp^t$.  Robinson \cite{rob1} addressed the question of when this yields a metaplectic-c prequantization for $(M_E,\omega_E)$ with the following theorem.

\begin{theorem}\label{thm:rob}
(Robinson, 1990) 
Assume that the symplectic reduction $(M_E,\omega_E)$ is a manifold, and that its symplectic frame bundle $\Sp(M_E,\omega_E)$ can be identified with the quotient of $\Sp(TS/TS^\perp)$ by the induced flow $\tp^t$.  If $\gamma_S$ has trivial holonomy over all of the closed orbits of $\tp^t$ on $\Sp(TS/TS^\perp)$, then the quotient of $(P_S,\gamma_S)$ by $\hp^t$ is a metaplectic-c prequantization for $(M_E,\omega_E)$.
\end{theorem}

To satisfy the holonomy condition in Theorem \ref{thm:rob}, it is sufficient -- though not necessary -- to show that $\gamma^S$ has trivial holonomy over all closed orbits of $\tp^t$ on $\Sp(M,\omega;S)$.  This is the quantization condition that was explored in the remainder of \cite{rob1}.  However, we argue that a more robust condition arises when we evaluate the holonomy over orbits in $\Sp(TS/TS^\perp)$.  As such, we propose the following definition.

\begin{definition}\label{def:quantE}
When $\gamma_S$ has trivial holonomy over all closed orbits of $\tp^t$ on $\Sp(TS/TS^\perp)$, we say that $E$ is a \textbf{quantized energy level} of the system $(M,\omega,H)$.
\end{definition}

In those cases where $E$ is a quantized energy level and the symplectic reduction at $E$ exists, Theorem \ref{thm:rob} gives a sufficient condition for the symplectic reduction to admit a metaplectic-c prequantization.  However, the given definition of a quantized energy level can also be applied to systems where the symplectic reduction does not exist.

The next section contains the proof of the dynamical invariance property.  We will show that if $S$ is a regular level set of two functions $H_1$ and $H_2$ on $M$, then $S$ corresponds to a quantized energy level for $(M,\omega,H_1)$ if and only if it does so for $(M,\omega,H_2)$.

\section{Dynamical Invariance}\label{sec:dynam}

We continue to use all of the definitions established in Section \ref{sec:energy}.  To prove that the quantized energy condition is dynamically invariant, we proceed in two stages.  In Section \ref{subsec:stage1}, we prove Theorem \ref{thm:stage1}, which states that if $E$ is a quantized energy level for the system $(M,\omega,H)$, then for any diffeomorphism $f:\R\rightarrow\R$, the value $f(E)$ is a quantized energy level for the system $(M,\omega,f\circ H)$.  While this is a less general statement, it allows us to establish some useful preliminary observations.  Then, in Section \ref{subsec:stage2}, we prove Theorem \ref{thm:stage2}, which states that if the functions $H_1,H_2:M\rightarrow\R$ have regular values $E_1$ and $E_2$, respectively, such that $H_1^{-1}(E_1)=H_2^{-1}(E_2)$, then $E_1$ is a quantized energy level for $(M,\omega,H_1)$ if and only if $E_2$ is a quantized energy level for $(M,\omega,H_2)$.

\subsection{Invariance Under a Diffeomorphism}\label{subsec:stage1}

Let $H:M\rightarrow\R$ be given, and fix a regular value $E$ of $H$.  Let $f:\R\rightarrow\R$ be a diffeomorphism, and let $S=H^{-1}(E)=(f\circ H)^{-1}(f(E))$.

Denote the Hamiltonian vector field for $H$ by $\xi_H$, and let it have flow $\phi^t$ on $M$.  Then $\xi_H$ lifts to $\txi_H$ on $\Sp(M,\omega)$ and to $\hxi_H$ on $P$, with flows $\tp^t$ and $\hp^t$, respectively, as in Section \ref{subsec:hamvt}.  Similarly, denote the Hamiltonian vector field for $f\circ H$ by $\xi_{fH}$, and let it have flow $\rho^t$ on $M$.  By the same process, we obtain the lifts $\txi_{fH}$ on $\Sp(M,\omega)$ and $\hxi_{fH}$ on $P$, with flows $\tr^t$ and $\hr^t$, respectively.

Observe that for any $m\in M$, $$d(f\circ H)|_{m}=\left.\dd{f(H)}{H}\right|_{H(m)}dH|_m,$$ which implies that $$\xi_{fH}(m)=\left.\dd{f(H)}{H}\right|_{H(m)}\xi_H(m).$$  Since $S$ is a level set of $H$, $\displaystyle\left.\dd{f(H)}{H}\right|_{H(m)}$ is constant over $S$.  Let $c\in\R$ be that constant value on $S$.  Then $\xi_{fH}=c\xi_H$ everywhere on $S$, which implies that $\rho^t=\phi^{ct}$ everywhere on $S$.  

Recall that $\Sp(TS/TS^\perp)$ can be naturally identified with the bundle associated to $\Sp(M,\omega;S)$ by the group homomorphism $\nu:\Sp(V;W)\rightarrow\Sp(W/W^\perp)$.  We write an element of $W/W^\perp$ as the equivalence class $[w]$ for some $w\in W$, and we write an element of $T_sS/T_sS^\perp$ as the equivalence class $[\zeta]$ for some $\zeta\in T_sS$.  As discussed in Section \ref{subsec:mpcred}, the lifted flow $\tp^t$ on $\Sp(M,\omega)$ maps $\Sp(M,\omega;S)$ to itself and induces a flow $\tp^t$ on $\Sp(TS/TS^\perp)$.  More explicitly, for any $b'\in\Sp(TS/TS^\perp)$, we define $\tp^t(b')$ by choosing $b\in\Sp(M,\omega;S)$ such that $[bw]=b'[w]$ for all $w\in W$, and setting $$\tp^t(b')[w]=[\tp^t(b)w],\ \ \forall w\in W.$$  These remarks apply equally well to $\tr^t$.

\begin{lemma}\label{lem:cvalue1}
$\tr^t=\tp^{ct}$ on $\Sp(TS/TS^\perp)$.
\end{lemma}

\begin{proof}
Fix $s\in S$ and $b'\in\Sp(TS/TS^\perp)_s$.  Let $b\in\Sp(M,\omega;S)_s$ be such that for all $w\in W$, $[bw]=b'[w]$.  Since $\rho^t=\phi^{ct}$ on $S$, we have $\rho_*^t|TS=\phi_*^{ct}|TS$.  Therefore, for any $w\in W$, $$\tr^t(b)w=\rho_*^t|_s(bw)=\phi^{ct}_*|_s(bw)=\tp^{ct}(b)w,$$ where we used the fact that $bw\in T_sS$.  Using the definitions of $\tp^t$ and $\tr^t$ on $\Sp(TS/TS^\perp)$, it follows that $$\tr^t(b')[w]=[\tr^t(b)w]=[\tp^{ct}(b)w]=\tp^{ct}(b')[w].$$  Thus $\tr^t=\tp^{ct}$ on $\Sp(TS/TS^\perp)$.
\end{proof}

Lemma \ref{lem:cvalue1} implies that $\txi_{fH}=c\txi_{H}$ on $\Sp(TS/TS^\perp)$.  Since $\txi_{fH}$ is just a constant multiple of $\txi_H$, it is clear that the flows of both vector fields have the same orbits.  Thus, if $\gamma_S$ has trivial holonomy over all closed orbits of $\tp^t$ on $\Sp(TS/TS^\perp)$, then it also has trivial holonomy over all closed orbits of $\tr^t$.  The following is now immediate.

\begin{theorem}\label{thm:stage1}
If $E$ is a quantized energy level for $(M,\omega,H)$ and $f:\R\rightarrow\R$ is a diffeomorphism, then $f(E)$ is a quantized energy level for $(M,\omega,f\circ H)$.
\end{theorem}

\subsection{Regular Level Set of Multiple Functions}\label{subsec:stage2}

Let $H_1,H_2:M\rightarrow\R$ be two smooth functions, and assume that there are $E_1,E_2\in\R$ such that $E_j$ is a regular value of $H_j$ for $j=1,2$, and $H_1^{-1}(E_1)=H_2^{-1}(E_2)$.  Let this shared level set be $S$.  Let the Hamiltonian vector fields $\xi_{H_1}$ and $\xi_{H_2}$ on $M$ have flows $\phi^t$ and $\rho^t$, respectively.  The induced flows $\tp^t$ and $\tr^t$ on $\Sp(TS/TS^\perp)$ are defined as in Section \ref{subsec:stage1}.  

The key observation in this scenario is that $TS$ and $TS^\perp$ do not depend on $H_1$ or $H_2$.  For all $s\in S$, we have $$T_sS^\perp=\mbox{span}\cb{\xi_{H_1}(s)}=\mbox{span}\cb{\xi_{H_2}(s)},$$ which implies that there is a map $c:S\rightarrow\R$ such that $\xi_{H_2}(s)=c(s)\xi_{H_1}(s)$ for all $s\in S$.  The map $c$ need not be constant on $S$, and $\xi_{H_1}$ and $\xi_{H_2}$ need not be parallel away from $S$.

In Lemma \ref{lem:cvalue1}, we were able to establish a relationship between $\tr^t$ and $\tp^t$.  In this case, we will not be able to relate $\tr^t$ and $\tp^t$ directly, but must instead work in terms of their time derivatives.  

For any $s\in S$, $\rho^t_*|_s$ and $\phi^t_*|_s$ preserve $TS$.  Therefore, for all $\zeta\in T_sS$, the tangent vectors $\lat{\dd{}{t}}{t=0}\lb{\rho^t_*|_s\zeta}$ and $\lat{\dd{}{t}}{t=0}\lb{\phi^t_*|_s\zeta}$ are elements of $T_\zeta TS$.   Note that the vector space $T_sS$ can be naturally identified with its tangent space $T_\zeta T_sS$, which is a subspace of $T_\zeta TS$.  Thus $\xi_{H_1}(s)\in T_sS$ can also be viewed as an element of $T_\zeta TS$.  This identification is used in the statement of the following result.

\begin{lemma}\label{lem:intuit}
For all $s\in S$ and all $\zeta\in T_sS$, $$\lat{\dd{}{t}}{t=0}\lb{\rho^t_*|_s\zeta}=c(s)\lat{\dd{}{t}}{t=0}\lb{\phi^t_*|_s\zeta}+(\zeta c)\xi_{H_1}(s).$$ 
\end{lemma}
\begin{proof}
We ignore $M$, and treat $S$ as a $(2n-1)$-dimensional manifold.  The vector fields $\xi_{H_1}$ and $\xi_{H_2}$ on $S$ have flows $\phi^t$ and $\rho^t$, respectively, and satisfy $\xi_{H_2}=c\xi_{H_1}$.

Fix $s\in S$, and let $U$ be a coordinate neighborhood for $s$ with coordinates $X=(X_1,\ldots,X_{2n-1})$.  We write $\phi^t=(\phi^t_1,\ldots,\phi^t_{2n-1})$ and $\rho^t=(\rho_1^t,\ldots,\rho_{2n-1}^t)$ with respect to the local coordinates.  Then $\phi_*^t|_s$ becomes a $(2n-1)\times(2n-1)$ matrix with the $(j,k)$th entry given by $$\lb{\phi^t_*|_s}_{jk}=\left.\pp{\phi^t_j(X)}{X_k}\right|_{X=s},$$ and similarly for $\rho^t_*|_s$.  Also, since $\xi_{H_2}=c\xi_{H_1}$, we have   $$\left.\dd{}{t}\right|_{t=0}\rho_j^t=c\left.\dd{}{t}\right|_{t=0}\phi_j^t,\ \ j=1,\ldots,2n-1,$$ everywhere on $U$.

Let $\zeta\in T_sS$ be arbitrary.  Then $\zeta=\sum_{k=1}^{2n-1}a_k\pp{}{X_k}$ for some coefficients $a_1,\ldots,a_{2n-1}\in\R$, and $$\lb{\phi^t_*|_s\zeta}_j=\left.\sum_{k=1}^{2n-1}a_k\pp{\phi^t_j(X)}{X_k}\right|_{X=s}=\zeta\phi_j^t.$$  The same process establishes that $(\rho^t_*|_{s}\zeta)_j=\zeta\rho_j^t$.  When we take the time derivative of the latter equation, we find $$\left.\dd{}{t}\right|_{t=0}\lb{\rho_*^t|_s\zeta}_j=\zeta\lb{\left.\dd{}{t}\right|_{t=0}\rho^t_j}=\zeta\lb{c\left.\dd{}{t}\right|_{t=0}\phi^t_j}.$$  Now we apply the Leibniz rule and recall that $\left.\dd{}{t}\right|_{t=0}\phi^t_j(s)=\xi_{H_1}(s)_j$.  The result is 
\begin{equation}\label{eq:result}
\left.\dd{}{t}\right|_{t=0}\lb{\rho_*^t|_s\zeta}_j=c(s)\left.\dd{}{t}\right|_{t=0}(\phi_*^t|_s\zeta)_j+(\zeta c)\xi_{H_1}(s)_j.
\end{equation}

The coordinates on $U$ naturally generate coordinates for $TS|_U$, which in turn provide a basis for $T_\zeta TS$.  In terms of that basis,  $$\lat{\dd{}{t}}{t=0}\lb{\rho^t_*|_s\zeta}=\lb{\xi_{H_2}(s)_1,\ldots,\xi_{H_2}(s)_{2n-1},\left.\dd{}{t}\right|_{t=0}\lb{\rho_*^t|_s\zeta}_1,\ldots,\left.\dd{}{t}\right|_{t=0}\lb{\rho_*^t|_s\zeta}_{2n-1}},$$ with the analogous expression for $\lat{\dd{}{t}}{t=0}\lb{\phi^t_*|_s\zeta}$, and  $$\xi_{H_1}(s)=(0,\ldots,0,\xi_{H_1}(s)_1,\ldots,\xi_{H_1}(s)_{2n-1})$$ as an element of $T_\zeta TS$.  Substituting Equation (\ref{eq:result}) into the expression for $\lat{\dd{}{t}}{t=0}\lb{\rho^t_*|_s\zeta}$ shows that  $$\lat{\dd{}{t}}{t=0}\lb{\rho^t_*|_s\zeta}=c(s)\lat{\dd{}{t}}{t=0}\lb{\phi^t_*|_s\zeta}+(\zeta c)\xi_{H_1}(s),$$ as desired.  
\end{proof}

\begin{lemma}\label{lem:cvalue2}
For all $s\in S$ and all $b'\in\Sp(TS/TS^\perp)_s$, $$\txi_{H_2}(b')=c(s)\txi_{H_1}(b').$$
\end{lemma}
\begin{proof}
Within the model vector space $V$, fix a symplectic basis $(\hx_1,\ldots,\hx_n,\hy_1,\ldots,\hy_n)$.  Let $$W=\mbox{span}\cb{\hx_1,\ldots,\hx_n,\hy_2,\ldots,\hy_n},$$ so that $$W^\perp=\mbox{span}\cb{\hx_1}\ \ \mbox{and}\ \ W/W^\perp=\mbox{span}\cb{[\hx_2],\ldots,[\hx_n],[\hy_2],\ldots,[\hy_n]}.$$  For convenience, let $(\hz_1,\ldots,\hz_{2n})=(\hx_1,\ldots,\hx_n,\hy_2,\ldots,\hy_n,\hy_1)$.  Then  $$W=\mbox{span}\cb{\hz_1,\ldots,\hz_{2n-1}},\ \ W^\perp=\mbox{span}\cb{\hz_1},\ \ W/W^\perp=\mbox{span}\cb{[\hz_2],\ldots,[\hz_{2n-1}]}.$$

For this proof, we will not view an element of the symplectic frame bundle as a map, but as a symplectic basis for a tangent space.  Explicitly, over $m\in M$, we identify $b\in\Sp(M,\omega)_m$ with the ordered $2n$-tuple $(\zeta_1,\ldots,\zeta_{2n})\in\lb{T_mM}^{2n}$, where $b\hz_j=\zeta_j$ for $j=1,\ldots,2n$.  Similarly, over $s\in S$, we identify $b'\in\Sp(TS/TS^\perp)_s$ with the ordered $(2n-2)$-tuple $([\zeta_2],\ldots,[\zeta_{2n}])\in\lb{T_sS/T_sS^\perp}^{2n-2}$, where $b'[\hz_j]=[\zeta_j]$ for $j=2,\ldots,2n-1$. 

Consider $s\in S$ and $b\in\Sp(M,\omega;S)_s$.  Let $b\hz_j=\zeta_j\in T_sM$ for $j=1,\ldots,2n$.  Then $\zeta_j\in T_sS$ for $j=1,\ldots,2n-1$ and $\zeta_1\in T_sS^\perp$.  The flow $\tr^t$ on $\Sp(M,\omega;S)$, evaluated at $b$, is given by $$\tr^t(b)=\rho^t_*|_m\circ b=\lb{\rho^t_*|_m\zeta_1,\ldots,\rho^t_*|_m\zeta_{2n}},$$ and the analogous expression holds for $\tp^t$.  

When we descend to $\Sp(TS/TS^\perp)$, the image of $b$ is the element $b'=([\zeta_2],\ldots,[\zeta_{2n-1}])\in(T_sS/T_sS^\perp)^{2n-2}$.  The induced flow $\tr^t$ on $\Sp(TS/TS^\perp)$, evaluated at $b'$, is $$\tr^t(b')=\lb{[\rho^t_*|_s\zeta_2],\ldots,[\rho^t_*|_s\zeta_{2n-1}]}.$$  As $t$ varies, this expression describes a curve in $(TS/TS^\perp)^{2n-2}$, through $b'$.  The tangent vector to this curve at $b'$ is 
\begin{align*}
\txi_{H_2}(b')=\lat{\dd{}{t}}{t=0}\tr^t(b')=&\lb{\lat{\dd{}{t}}{t=0}[\rho^t_*|_s\zeta_2],\ldots,\lat{\dd{}{t}}{t=0}[\rho^t_*|_s\zeta_{2n-1}]}\\
&\rule{0.5in}{0in}\in T_{[\zeta_2]}(TS/TS^\perp)\times\ldots\times T_{[\zeta_{2n-1}]}(TS/TS^\perp).
\end{align*}

The pushforward of the quotient map $TS\rightarrow TS/TS^\perp$, based at $\zeta_j$, is a linear surjection $T_{\zeta_j}TS\rightarrow T_{[\zeta_j]}(TS/TS^\perp)$ whose kernel is $T_{\zeta_j} T_sS^\perp$.  If we identify $T_{\zeta_j}T_sS^\perp$ with $T_sS^\perp$, then we find a natural isomorphism between $T_{[\zeta_j]}(TS/TS^\perp)$ and $T_{\zeta_j}TS/T_sS^\perp$.  Applying that isomorphism to each component of $\txi_{H_2}(b')$ yields  
\begin{align*}
\txi_{H_2}(b')=&\lb{\ls{\lat{\dd{}{t}}{t=0}(\rho^t_*|_s\zeta_2)},\ldots,\ls{\lat{\dd{}{t}}{t=0}(\rho^t_*|_s\zeta_{2n-1})}}\\
&\rule{0.5in}{0in}\in T_{\zeta_2}TS/T_sS^\perp\times\ldots\times T_{\zeta_{2n-1}}TS/T_sS^\perp.
\end{align*}
Since $\zeta_j\in T_sS$ for $j=2,\ldots,2n-1$, we can use Lemma \ref{lem:intuit}:  $$\lat{\dd{}{t}}{t=0}\lb{\rho^t_*|_s\zeta_j}=c(s)\lat{\dd{}{t}}{t=0}\lb{\phi^t_*|_s\zeta_j}+(\zeta_jc)\xi_{H_1}(s),$$ where each term in the above equation is viewed as an element of $T_{\zeta_j}TS$.  Upon descending to $T_{\zeta_j}TS/T_sS^\perp$, the multiple of $\xi_{H_1}(s)$ vanishes and we find that $$\ls{\lat{\dd{}{t}}{t=0}\lb{\rho^t_*|_s\zeta_j}}=c(s)\ls{\lat{\dd{}{t}}{t=0}\lb{\phi^t_*|_s\zeta_j}}.$$  Therefore
\begin{align*}
\txi_{H_2}(b')&=\lb{c(s)\ls{\lat{\dd{}{t}}{t=0}(\phi^t_*|_s\zeta_2)},\ldots,c(s)\ls{\lat{\dd{}{t}}{t=0}(\phi^t_*|_s\zeta_{2n-1})}}\\
&=c(s)\txi_{H_1}(b').
\end{align*}
This completes the proof.
\end{proof}

Thus $\txi_{H_1}$ and $\txi_{H_2}$ are parallel on $\Sp(TS/TS^\perp)$, although the multiplicative factor that relates them is not necessarily constant over $S$.  This implies that $\tp^t$ and $\tr^t$ have identical orbits in $\Sp(TS/TS^\perp)$, and if $\gamma_S$ has trivial holonomy over the closed orbits of $\tp^t$, then the same must also be true for $\tr^t$.  We can now conclude the result that was the objective of this section.

\begin{theorem}\label{thm:stage2}
If $H_1,H_2:M\rightarrow\R$ are smooth functions such that $H_1^{-1}(E_1)=H_2^{-1}(E_2)$ for regular values $E_j$ of $H_j$, $j=1,2$, then $E_1$ is a quantized energy level for $(M,\omega,H_1)$ if and only if $E_2$ is a quantized energy level for $(M,\omega,H_2)$.
\end{theorem}

\section{The Harmonic Oscillator}\label{sec:sho}

In this section, we apply our quantized energy condition to the $n$-dimensional harmonic oscillator.  The metaplectic-c quantization of this system has already been examined in \cite{rob1} and \cite{rr1}, so it will come as no surprise that we obtain the correct quantized energy levels:  $$E_N=\hbar\lb{N+\frac{n}{2}},\ \ N\in\Z.$$  Our first objective in studying this example is to present a useful computational technique:  we will locally change coordinates from Cartesian to symplectic polar on the base manifold, then show how to lift this change to the symplectic frame bundle and prequantization bundle.  Once symplectic polar coordinates have been established, we will turn to our second objective, which is to construct an explicit example that illustrates the principal of dynamical invariance.

\subsection{Initial Choices}\label{subsec:shoinit}

Let $(\hx_1,\ldots,\hx_n,\hy_1,\ldots,\hy_n)$ be a basis for the model vector space $V$ such that $\Omega=\sum_{j=1}^n\hx_j^*\wedge\hy_j^*$.  All elements of $V$ will be written as ordered $2n$-tuples with respect to this basis.  Assume that $V$ is identified with $\C^n$ by mapping the point $(a_1,\ldots,a_n,b_1,\ldots,b_n)\in V$ to the point $(b_1+ia_1,\ldots,b_n+ia_n)\in\C^n$.  Then the complex structure $J$ on $V$ is given in matrix form by $J=\four{0}{I}{-I}{0}$, where $I$ is the $n\times n$ identity matrix.

A note concerning notation:  given expressions $a_j$ and $b_j$, $j=1,\ldots,n$, we let $\lb{a_j}_{1\leq j\leq n}$ represent the $n\times n$ diagonal matrix $\mbox{diag}(a_1,\ldots,a_n)$, and we let $\two{a_j}{b_j}_{1\leq j\leq n}$ represent the $2n\times 1$ column vector $(a_1,\ldots,a_n,b_1,\ldots,b_n)^T$.

Let $M=\R^{2n}$, with Cartesian coordinates $(p_1,\ldots,p_n,q_1,\ldots,q_n)$ and symplectic form $\omega=\sum_{j=1}^n dp_j\wedge dq_j$.  The energy function for the harmonic oscillator is  $H=\frac{1}{2}\sum_{j=1}^n\lb{p_j^2+q_j^2}$, which has Hamiltonian vector field  $$\xi_{H}=\sum_{j=1}^n\lb{q_j\pp{}{p_j}-p_j\pp{}{q_j}}.$$  Let this vector field have flow $\phi^t$ on $M$.  

There is a global trivialization of $TM$ given by identifying $\hx_j\mapsto\left.\pp{}{p_j}\right|_{m}$, $\hy_j\mapsto\left.\pp{}{q_j}\right|_{m}$ at every point $m\in M$, which induces a global trivialization of $\Sp(M,\omega)$.  
Let $P$ be the trivial bundle $M\times\Mp^c(V)$, with bundle projection map $P\maps{\Pi}M$.  Define the map $P\maps{\Sigma}\Sp(M,\omega)$ by $$\Sigma(m,a)=(m,\sigma(a)),\ \ \forall m\in M,a\in\Mp^c(V),$$ where the ordered pair on the right-hand side is written with respect to the global trivialization stated above.  Then $(P,\Sigma)$ is a metaplectic-c structure for $(M,\omega)$.  

On $M$, define the one-form $\beta$ by $$\beta=\frac{1}{2}\sum_{j=1}^n(p_jdq_j-q_jdp_j),$$ so that $d\beta=\omega$.  Let $\vartheta_0$ be the trivial connection on the product bundle $M\times\Mp^c(V)$, and define the one-form $\gamma$ on $P$ by $$\gamma=\frac{1}{i\hbar}\Pi^*\beta+\frac{1}{2}\eta_*\vartheta_0.$$  Then $(P,\Sigma,\gamma)$ is a metaplectic-c prequantization for $(M,\omega)$.  In fact, it is the unique metaplectic-c prequantization up to isomorphism, since $M$ is contractible.

Fix $E>0$, a regular value of $H$, and let $S=H^{-1}(E)$.  Let $m_0\in S$ be given.  Observe that $H$ and $\omega$ are invariant under following transformations:  (1) rotation of the $p_jq_j$-plane about the origin by any angle, for any $j$, and (2) simultaneous rotations of the $p_jp_k$- and $q_jq_k$-planes about the origin by the same angle, for any $j\neq k$.  Thus, after suitable rotations, we can assume that $m_0=(p_{01},\ldots,p_{0n},0,\ldots,0)$ in Cartesian coordinates, where $p_{0j}\neq 0$ for all $j$.  It is easily established that $$\phi^t(m_0)=\four{(\cos t)_{1\leq j\leq n}}{(\sin t)_{1\leq j\leq n}}{(-\sin t)_{1\leq j\leq n}}{(\cos t)_{1\leq j\leq n}}\two{p_{0j}}{0}_{1\leq j\leq n}.$$  This expression describes a periodic orbit with period $2\pi$.  Let $\sC$ be the orbit of $\xi_H$ through $m_0$:  $\sC=\cb{\phi^t(m_0):t\in\R}$.

\subsection{From Cartesian to Symplectic Polar Coordinates}\label{subsec:cvars}

\subsubsection{On the Manifold $M$}

By symplectic polar coordinates, we mean the local coordinates $(s_1,\ldots,s_n,\theta_1,\ldots,\theta_n)$ given by 
\begin{equation}\label{eq:ctop}
s_j=\frac{1}{2}\lb{p_j^2+q_j^2},\ \ \theta_j=\tan^{-1}\lb{\frac{q_j}{p_j}},\ \ j=1,\ldots,n,
\end{equation}
whenever these expressions are defined.  The polar angles $\theta_j$ are all defined modulo $2\pi$.  For later reference, the inverse coordinate transformations are 
\begin{equation}\label{eq:ptoc}
p_j=\sqrt{2s_j}\cos\theta_j,\ \ q_j=\sqrt{2s_j}\sin\theta_j,\ \ j=1,\ldots,n.
\end{equation}  Let $$U=\cb{(p_1,\ldots,p_n,q_1,\ldots,q_n)\in M:p_j^2+q_j^2>0,\ j=1,\ldots,n},$$ so that symplectic polar coordinates and the corresponding vector fields $\cb{\pp{}{s_1},\ldots,\pp{}{s_n},\pp{}{\theta_1},\ldots,\pp{}{\theta_n}}$ are defined everywhere on $U$.  Observe that the orbit $\sC$ is contained in $U$.  We will construct a local trivialization for $\Sp(M,\omega)$ over $U$, and a local trivialization for $P$ over $\sC$.  

When we convert to symplectic polar coordinates on $U$, we find $$\omega=\sum_{j=1}^n ds_j\wedge d\theta_j,$$ which implies that for all $m\in U$, $\cb{\lat{\pp{}{s_1}}{m},\ldots,\lat{\pp{}{s_n}}{m},\lat{\pp{}{\theta_1}}{m},\ldots,\lat{\pp{}{\theta_n}}{m}}$ is a symplectic basis for $T_mM$.  Further, 
\begin{equation}\label{eq:miscpolar}
\beta=\sum_{j=1}^n s_jd\theta_j,\ \ H=\sum_{j=1}^n s_j,\ \ \xi_{H}=-\sum_{j=1}^n\pp{}{\theta_j},
\end{equation}
and $m_0=(s_{01},\ldots,s_{0n},0,\ldots,0)$, where $s_{0j}=\frac{1}{2}p_{0j}^2$, $j=1,\ldots,n$.  The curve $\sC$ is $$\sC=\cb{(s_{01},\ldots,s_{0n},\tau,\ldots,\tau):\tau\in\R/2\pi\Z}.$$

The goal is to lift this change of coordinates to the symplectic frame bundle, then to the metaplectic-c prequantization.  To facilitate this process, we introduce the following notation.  Let $\Phi_c:U\rightarrow\R^{2n}$ be the Cartesian coordinate map, and let $U_c=\Phi_c(U)$.  Similarly, let $\Phi_p:U\rightarrow\R^{n}\times(R/2\pi\Z)^n$ be the symplectic polar coordinate map, and let $U_p=\Phi_p(U)$.  Let $F$ denote the transition map $\Phi_p\circ\Phi_c^{-1}$.  Then we have the following commutative diagram.
\begin{center}
$\xymatrix{
& U \ar[ld]_{\Phi_c} \ar[rd]^{\Phi_p} & \\
**[r]U_c\subset\R^{2n} \ar[rr]^F & & **[r]U_p\subset\R^{n}\times(\R/2\pi\Z)^n
}$
\end{center}
This will serve as a model for the changes of coordinates on $\Sp(M,\omega)$ and $P$.

\subsubsection{On the Symplectic Frame Bundle $\Sp(M,\omega)$}

Recall that an element of the fiber $\Sp(M,\omega)_m$ is a symplectic isomorphism from $V$ to $T_mM$.  Over $U$, we define the sections $b_c$ and $b_p$ of $\Sp(M,\omega)$ as follows.  At each $m\in U$, they are given by
\begin{eqnarray*}
b_c(m):V\rightarrow T_mM\ \ \mbox{such that}\ \ \two{\hx_j}{\hy_j}_{1\leq j\leq n}\mapsto\two{\lat{\pp{}{p_j}}{m}}{\lat{\pp{}{q_j}}{m}}_{1\leq j\leq n},\\
b_p(m):V\rightarrow T_mM\ \ \mbox{such that}\ \ \two{\hx_j}{\hy_j}_{1\leq j\leq n}\mapsto\two{\lat{\pp{}{s_j}}{m}}{\lat{\pp{}{\theta_j}}{m}}_{1\leq j\leq n}.
\end{eqnarray*}  
Using these sections, we define two maps $\widetilde{\Phi}_c$, $\widetilde{\Phi}_p$ on $\Sp(M,\omega)$ over $U$:
\begin{eqnarray*}
\widetilde{\Phi}_c:\Sp(M,\omega)|_U\rightarrow U_c\times\Sp(V)\ &\mbox{such that}&\ \widetilde{\Phi}_c(b_c(m)\cdot g)=(\Phi_c(m),g),\ \ \forall m\in U,\ \forall g\in\Sp(V),\\
\widetilde{\Phi}_p:\Sp(M,\omega)|_U\rightarrow U_p\times\Sp(V)\ &\mbox{such that}&\ \widetilde{\Phi}_p(b_p(m)\cdot g)=(\Phi_p(m),g),\ \ \forall m\in U,\ \forall g\in\Sp(V).
\end{eqnarray*}
In other words, the section $b_c$ determines the Cartesian trivialization of $\Sp(M,\omega)$ over $U$, and $b_p$ determines the symplectic polar trivialization.  Lifting the change of coordinates to $\Sp(M,\omega)|_{U}$ is equivalent to finding a map $\widetilde{F}:U_c\times\Sp(V)\rightarrow U_p\times\Sp(V)$ that is a lift of $F:U_c\rightarrow U_p$ and such that the following commutes.
\begin{center}
$\xymatrix{
& \Sp(M,\omega)|_{U} \ar[ld]_{\widetilde{\Phi}_c} \ar[rd]^{\widetilde{\Phi}_p} & \\
U_c\times\Sp(V) \ar[rr]^{\widetilde{F}} & & U_p\times\Sp(V)
}$
\end{center}
If such a map $\widetilde{F}$ exists, then it must satisfy $$\widetilde{F}\circ\widetilde{\Phi}_c(b_c)=\widetilde{\Phi}_p(b_c),$$ and indeed, once this condition is satisfied, then the value of $\widetilde{F}$ everywhere else will be determined by the $\Sp(V)$ group action and the fact that $\widetilde{F}$ is a lift of $F$.  Let us then evaluate $\widetilde{\Phi}_p(b_c)$.  

For all $m\in U$, let $$G(m)=\four{\lb{\lat{\pp{p_j}{s_j}}{m}}_{1\leq j\leq n}}{\lb{\lat{\pp{q_j}{s_j}}{m}}_{1\leq j\leq n}}{\lb{\lat{\pp{p_j}{\theta_j}}{m}}_{1\leq j\leq n}}{\lb{\lat{\pp{q_j}{\theta_j}}{m}}_{1\leq j\leq n}},$$ so that $$G(m)\two{\lat{\pp{}{p_j}}{m}}{\lat{\pp{}{q_j}}{m}}_{1\leq j\leq n}=\two{\lat{\pp{}{s_j}}{m}}{\lat{\pp{}{\theta_j}}{m}}_{1\leq j\leq n}.$$
Then $G(m)$ is a symplectic matrix, and can be viewed as an element of $\Sp(V)$.  From the definitions of $b_c(m)$ and the group action, we see that 
$$b_c(m):G(m)\two{\hx_j}{\hy_j}_{1\leq j\leq n}\mapsto G(m)\two{\lat{\pp{}{p_j}}{m}}{\lat{\pp{}{q_j}}{m}}_{1\leq j\leq n}=\two{\lat{\pp{}{s_j}}{m}}{\lat{\pp{}{\theta_j}}{m}}_{1\leq j\leq n}.$$  Thus $$\widetilde{\Phi}_p(b_c(m))=(\Phi_p(m),G(m)),$$ and more generally, $$\widetilde{\Phi}_p(b_c(m)\cdot g)=(\Phi_p(m),G(m)g),\ \ \forall g\in\Sp(V).$$  Hence we define the map $\widetilde{F}$ by $$\widetilde{F}(\Phi_c(m),g)=(\Phi_p(m),G(m)g),\ \ \forall m\in U,\forall g\in\Sp(V).$$  

Observe, using Equation (\ref{eq:ptoc}), that the entries of $G(m)$ are singly defined with respect to the angles $\theta_j$, so $G(m)$ is singly defined as $m$ traverses the curve $\sC$, or any other closed path through $U$.  This demonstrates that we can change coordinates from Cartesian to polar everywhere on $\Sp(M,\omega)|_U$.

\subsubsection{On the Metaplectic-c Prequantization $P$}

When we lift the change of variables to $P$, we restrict our attention further from $U$ to $\sC$.  Let $\sC_c=\Phi_c(\sC)$ and let $\sC_p=\Phi_p(\sC)$.  To emphasize that all of our calculations take place over the closed curve $\sC$, we let $m(\tau)=(s_{01},\ldots,s_{0n},\tau,\ldots,\tau)\in\sC$ for all $\tau\in\R/2\pi\Z$, and we abbreviate $G(m(\tau))$ by $G(\tau)$.  Then $G(\tau)$ is a closed loop through $\Sp(V)$ with period $2\pi$.

Recall that $P=M\times\Mp^c(V)$, and that $\Sigma:P\rightarrow\Sp(M,\omega)$ is given by $\Sigma(m,a)=(m,\sigma(a))$ for all $(m,a)\in P$, where the right-hand side is written with respect to the Cartesian trivialization.  In other words, $P$ was constructed to be consistent with Cartesian coordinates.  To formalize this property, we define $$\hat{\Phi}_c:P|_{\sC}\rightarrow\sC_c\times\Mp^c(V)\ \ \mbox{such that}\ \ \hat{\Phi}_c(m(\tau),a)=(\Phi_c(m(\tau)),a),\ \ \forall (m(\tau),a)\in P|_{\sC}.$$  Then $\hat{\Phi}_c$ is compatible with $\widetilde{\Phi}_c$ in the sense that the following diagram commutes.
\begin{center}
$\xymatrix{
P|_{\sC} \ar[r]^{\Sigma} \ar[d]^{\hat{\Phi}_c}  & \Sp(M,\omega)|_{\sC} \ar[d]^{\widetilde{\Phi}_c} \\
\sC_{c}\times\Mp^c(V) \ar[r]^{\sigma} & \sC_c\times\Sp(V)
}$
\end{center}
In the bottom line, $\sigma$ maps $\sC_c\times\Mp^c(V)$ to $\sC_c\times\Sp(V)$ by acting on the second component.

Our goal is to construct a map $\hat{\Phi}_p:P|_{\sC}\rightarrow \sC_p\times\Mp^c(V)$ that is compatible with $\widetilde{\Phi}_p$ in the same sense.  To do so, we will find a map $\hat{F}$ such that the diagram below commutes, then set $\hat{\Phi}_p=\hat{F}\circ\hat{\Phi}_c$.
\begin{center}
$\xymatrix{
& P|_{\sC} \ar[dl]_{\hat{\Phi}_c} \ar[rd] \ar[rrrd]^{\Sigma} & \\
\sC_c\times\Mp^c(V) \ar[rr]^{\hat{F}} \ar[rrrd]^{\sigma} & & \sC_p\times\Mp^c(V) \ar[rrrd]^(.25){\sigma} & & \Sp(M,\omega)|_{\sC} \ar[ld]_(.4){\widetilde{\Phi}_c} \ar[rd]^(.4){\widetilde{\Phi}_p} & \\
& & & \sC_c\times\Sp(V) \ar[rr]^{\widetilde{F}} & & \sC_p\times\Sp(V)
}$
\end{center} 

Since $\hat{F}$ must be a lift of $\widetilde{F}$ and therefore of $F$, we assume that $\hat{F}$ takes the form $$\hat{F}(\Phi_c(m(\tau)),a)=(\Phi_p(m(\tau)),\hat{G}(\tau)a),\ \ \forall m(\tau)\in\sC,\forall a\in\Mp^c(V),$$ where $\hat{G}(\tau)\in\Mp^c(V)$.  From the condition that $$\sigma\circ\hat{F}\circ\hat{\Phi}_c=\widetilde{F}\circ\sigma\circ\hat{\Phi}_c,$$ it follows that we must have $\sigma(\hat{G}(\tau))=G(\tau)$ for all $\tau\in\R/2\pi\Z$.  That is, $\hat{G}(\tau)$ must be a lift of $G(\tau)$ to $\Mp^c(V)$.  More specifically, in order for $\hat{F}$ to be singly defined, $\hat{G}(\tau)$ must be a closed loop in $\Mp^c(V)$ with period $2\pi$.  

There is another consideration:  the effect of this change of variables on the one-form $\gamma$.  Recall that $\gamma$ is defined on $P$ by $$\gamma=\frac{1}{i\hbar}\Pi^*\beta+\frac{1}{2}\eta_*\vartheta_0,$$ where $\vartheta_0$ is the trivial connection on the product bundle $M\times\Mp^c(V)$.  When we change to symplectic polar coordinates, we would like $\gamma$ to retain the same form, where $\beta$ can be written in polar coordinates as in Equation (\ref{eq:miscpolar}), and where $\vartheta_0$ now represents the trivial connection on the product bundle $\sC_p\times\Mp^c(V)$.  Since $\ker(\eta)=\Mp(V)$, we can accomplish this by requiring that the path $\hat{G}(\tau)$ lie within $\Mp(V)$.  

Recall the parametrization of $\Mp^c(V)$ that was described in Section \ref{subsec:vtspc}.  If $\hat{G}(\tau)$ is a lift of $G(\tau)$ to $\Mp(V)$, then the parameters of $\hat{G}(\tau)$ have the form $(G(\tau),\mu(\tau))$ where $\mu(\tau)^2\Det_\C C_{G(\tau)}=1$.  To determine $\mu(\tau)$, we must examine the matrix $G(\tau)\in\Sp(V)$ explicitly.  Using Equation (\ref{eq:ptoc}), we find $$G(\tau)=\four{\lb{\frac{1}{\sqrt{2s_{0j}}}\cos\tau}_{1\leq j\leq n}}{\lb{\frac{1}{\sqrt{2s_{0j}}}\sin\tau}_{1\leq j\leq n}}{\lb{-\sqrt{2s_{0j}}\sin\tau}_{1\leq j\leq n}}{\lb{\sqrt{2s_{0j}}\cos\tau}_{1\leq j\leq n}}.$$  Now we calculate $$C_{G(\tau)}=\frac{1}{2}(G(\tau)-JG(\tau)J)$$ using the complex structure noted in Section \ref{subsec:shoinit}, then convert to an $n\times n$ complex matrix.  The result is the diagonal matrix
$$C_{G(\tau)}=\frac{1}{2}\lb{\lb{\sqrt{2s_{0j}}+\frac{1}{\sqrt{2s_{0j}}}}e^{i\tau}}_{1\leq j\leq n},$$
 which has complex determinant $$\Det_\C C_{G(\tau)}=\prod_{j=1}^n\frac{1}{2}\lb{\sqrt{2s_{0j}}+\frac{1}{\sqrt{2s_{0j}}}}e^{i\tau}=Ke^{in\tau},$$ where $K$ is a positive real value that is constant over $\sC$.  Thus, a lift of $G(\tau)$ to $\Mp(V)$ is parametrized by 
\begin{equation}\label{eq:ghat}
\hat{G}(\tau)\mapsto\lb{G(\tau),\frac{1}{\sqrt{K}}e^{-in\tau/2}}.
\end{equation}

Notice that $(G(\tau+2\pi),\mu(\tau+2\pi))=(G(\tau),\mu(\tau)e^{-in\pi})$.  When $n$ is even, we get a closed loop through $\Mp(V)$ with period $2\pi$, which was the optimal outcome.  When $n$ is odd, however, traversing $\sC$ once multiplies $\mu(\tau)$ by $-1$.  This shows that we cannot always lift $G(\tau)$ to a closed path through $\Mp(V)$.
Nevertheless, we choose the path $\hat{G}(\tau)$ as defined in Equation (\ref{eq:ghat}) because it preserves the form of $\gamma$.  If we let $\varepsilon_n\in\Mp^c(V)$ be the element whose parameters are $(I,e^{-i\pi n})$, then $\hat{G}(\tau+2\pi)=\varepsilon_n\hat{G}(\tau)$.  

To prevent $\hat{F}$ from being multi-valued, let $\dot{\sC}=\cb{m(\tau)\in\sC:\tau\in(0,2\pi)}$, and let $\dot{\sC}_c$ and $\dot{\sC}_p$ be the images of $\dot{\sC}$ in Cartesian and symplectic polar coordinates, respectively.  Then let $\hat{F}:\dot{\sC}_c\times\Mp^c(V)\rightarrow\dot{\sC}_p\times\Mp^c(V)$ be given by $$\hat{F}(\Phi_c(m(\tau)),a)=(\Phi_p(m(\tau)),\hat{G}(\tau)a),\ \ \forall (m(\tau),a)\in P|_{\dot{\sC}},$$ and identify $P|_{\dot{\sC}}$ with $\dot{\sC}_p\times\Mp^c(V)$ under the map $\hat{\Phi}_p=\hat{F}\circ\hat{\Phi}_c$.  This gives us symplectic polar coordinates for $P$ over $\dot{\sC}$, and we will manually adjust for the fact that closing the loop multiplies $\hat{G}(\tau)$ by $\varepsilon_n$.  On $\dot{\sC}_p\times\Mp^c(V)$, $\gamma$ takes the form $$\gamma=\frac{1}{i\hbar}\sum_{j=1}^ns_{0j}d\theta_j+\frac{1}{2}\eta_*\vartheta_0,$$ where $\vartheta_0$ is the trivial connection on the product bundle, and where the values $s_{0j}$ are constant over $\sC_p$.

\subsection{Quantized Energy Levels of the Harmonic Oscillator}\label{subsec:elevels}

Having established the change of coordinates, we will now drop the explicit use of $\Phi_p$, $F$ and their lifts.  Elements of $U$ will be written with respect to the symplectic polar coordinates $(s_1,\ldots,s_n,$ $\theta_1,\ldots,\theta_n)$, which for convenience we abbreviate by $X_k$, $k=1,\ldots,2n$.  Using the identifications $\hx_j\mapsto\lat{\pp{}{s_j}}{m},\hy_j\mapsto\lat{\pp{}{\theta_j}}{m}$ for all $m\in U$, $j=1,\ldots,n$, we write $$\Sp(M,\omega)|_U=U\times\Sp(V)\ \ \mbox{and}\ \ P|_{\dot{\sC}}=\dot{\sC}\times\Mp^c(V).$$ 
At the level of tangent spaces, we have $$T_{(m,I)}\Sp(M,\omega)=T_mM\times\mathfrak{sp}(V),\ \ \forall m\in U,$$ and $$T_{(m,I)}P=T_mM\times\mathfrak{mp}^c(V)=T_mM\times\lb{\mathfrak{sp}(V)\oplus\mathfrak{u}(1)},\ \ \forall m\in\dot{\sC}.$$

Recall that on $U$, $\xi_{H}=-\sum_{j=1}^n\pp{}{\theta_j}$.  Therefore $T_mS^\perp=\mbox{span}\cb{\sum_{j=1}^n\lat{\pp{}{\theta_j}}{m}}$ at each $m\in U$.  Within the model vector space $V$, let $$W^\perp=\mbox{span}\cb{\hy_1+\ldots+\hy_n}.$$  Then $$W=\mbox{span}\cb{\hx_1-\hx_2,\ldots,\hx_{n-1}-\hx_n,\hy_1,\ldots,\hy_n},$$ and $$W/W^\perp=\mbox{span}\cb{[\hx_1-\hx_2],\ldots,[\hx_{n-1}-\hx_n],[\hy_1-\hy_2],\ldots,[\hy_{n-1}-\hy_n]}.$$  The local trivialization of $\Sp(M,\omega)|_U$ induces the local trivializations $\Sp(M,\omega;S)|_{U\cap S}=U\cap S\times\Sp(V;W)$ and $\Sp(TS/TS^\perp)|_{U\cap S}=U\cap S\times\Sp(W/W^\perp)$.

By definition, $P_S$ is the bundle associated to $P^S$ by the group homomorphism $\hat{\nu}:\Mp^c(V)\rightarrow\Mp^c(W/W^\perp)$.  Over $\dot{\sC}$, we have the local trivializations $P|_{\dot{\sC}}=\dot{\sC}\times\Mp^c(V)$, $P^S|_{\dot{\sC}}=\dot{\sC}\times\Mp^c(V;W)$ and $P_S|_{\dot{\sC}}=\dot{\sC}\times\Mp^c(W/W^\perp)$, where the associated bundle map $\hat{\nu}:P^S|_{\dot{\sC}}\rightarrow P_S|_{\dot{\sC}}$ acts by $$\hat{\nu}(m,a)=(m,\hat{\nu}(a)),\ \ \forall(m,a)\in P^S|_{\dot{\sC}}.$$  Using properties of $\hat{\nu}$ noted in Section \ref{subsec:vtsubspc}, we have $\hat{\nu}(\varepsilon_n)=\varepsilon_n\in\Mp^c(W/W^\perp)$ for all $n$.  Therefore, within $P_S$, the same factor of $\varepsilon_n$ is required to close the loop from $\dot{\sC}$ to $\sC$.  Further, since $\eta_*\circ\hat{\nu}_*=\eta_*$, the one-form $\gamma_S$ induced on $P_S$ does not change form under the map $\hat{\nu}$:  $$\gamma_S|_{\dot{\sC}}=\frac{1}{i\hbar}\sum_{j=1}^ns_{0j}d\theta_j+\frac{1}{2}\eta_*\vartheta_0,$$ where $\vartheta_0$ is now the trivial connection on the product bundle $\dot{\sC}\times\Mp^c(W/W^\perp)$.  Lastly, the relevant tangent spaces are $$T_{(m,I)}\Sp(TS/TS^\perp)=T_mM\times\mathfrak{sp}(W/W^\perp),\ \forall m\in U,$$ and $$T_{(m,I)}P_S=T_mM\times(\mathfrak{sp}(W/W^\perp)\oplus\mathfrak{u}(1)),\ \ \forall m\in\dot{\sC}.$$

Recall that we chose the initial point $m_0=(s_{01},\ldots,s_{0n},0,\ldots,0)$, and that $\sC$ is the orbit of $\xi_H$ through $m_0$.  We will show that the orbit of $\txi_H$ through $(m_0,I)$ is closed in $\Sp(TS/TS^\perp)$.  We first need to lift $\phi^t$ to the flow $\tp^t$ on $\Sp(M,\omega)$.  By definition, for any $m\in U$, $\tp^t(m,I)=(\phi^t(m),\phi^t_*|_{m})$.  This implies that  $$\txi_{H}(m,I)=\left.\dd{}{t}\right|_{t=0}\tp^t(m,I)=\lb{\xi_{H}(m),\left.\dd{}{t}\right|_{t=0}\phi^t_*|_{m}}.$$

We write $\phi^t=(\phi_1^t,\ldots,\phi^t_{2n})$ with respect to the symplectic polar coordinates.  Then $\phi_*^t|_{m}$ is a $2n\times 2n$ matrix, which we interpret as an element of $\Sp(V)$, and its components are given by $(\phi_*^t)_{jk}=\lat{\pp{\phi^t_j}{X_k}}{m}$.  Noting that $\lat{\dd{}{t}}{t=0}\phi^t_j=\lb{\xi_H}_j$, we compute
\begin{equation}\label{eq:dvfld}
\lb{\lat{\dd{}{t}}{t=0}\phi_*^t|_{m}}_{jk}=\left.\dd{}{t}\right|_{t=0}\left.\pp{}{X_k}\right|_{m}\phi^t_j=\left.\pp{}{X_k}\right|_{m}(\xi_{H})_j.
\end{equation}
But $\xi_H=-\sum_{j=1}^n\pp{}{\theta_j}$ on $U$, which has constant components, so  $\left.\dd{}{t}\right|_{t=0}\phi^t_*$ is identically $0$.  Thus $$\txi_{H}(m,I)=(\xi_{H}(m),0),\ \ \forall m\in U.$$  In particular, since the $\mathfrak{sp}(V)$ component of $\txi_{H}$ is constant over the orbit $\sC$, we can find the integral curve for $\txi_H$ on $\Sp(M,\omega)$ through $(m_0,I)$ by exponentiating:  it is simply $$\tp^t(m_0,I)=(\phi^t(m_0),I).$$  This is clearly a closed orbit with period $2\pi$.  The induced flow on the bundle $\Sp(TS/TS^\perp)$ also takes the form $$\tp^t(m_0,I)=(\phi^t(m_0),I).$$  

Now we lift $\txi_{H}$ to $P_S$, horizontally with respect to $\gamma_S$.  Over the curve $\dot{\sC}$, we calculate that $\xi_{H}\contr\beta=-\sum_{j=1}^ns_{0j}=-E$.  Then, for any $m\in\dot{\sC}$, the horizontal lift of $\txi_H$ to $T_{(m,I)}P_S$ is $$\hxi_{H}(m,I)=\lb{\xi_{H}(m),0\oplus\frac{E}{i\hbar}}.$$  The $\mathfrak{mp}^c(V)$ component is constant over $\dot{\sC}$, so the integral curve is again calculated by exponentiating: $$\hp^t(m_0,I)=\lb{\phi^t(m_0),\exp\lb{0\oplus\frac{Et}{i\hbar}}}=(\phi^t(m_0),e^{Et/i\hbar}),$$ where $e^{Et/i\hbar}\in U(1)\subset\Mp^c(W/W^\perp)$.  

The quantized energy levels are those values of $E$ for which this orbit closes.  However, we must remember that one circuit about $\sC$ introduces an extra factor of $\varepsilon_n$.  Referring to the properties of the parametrization stated in Section \ref{subsec:vtspc}, we see that $\hp^0(m_0,I)=\hp^{2\pi}(m_0,I)$ if and only if $e^{2\pi E/i\hbar}=e^{-in\pi}$, or in other words, when $\frac{2\pi E}{i\hbar}+in\pi=-2\pi i N$ for some $N\in\Z$.  Upon rearranging, we find $$E=\hbar\lb{N+\frac{n}{2}},\ \ N\in\Z.$$  Thus the quantization condition correctly reproduces the expected energy levels of the harmonic oscillator.

\subsection{Example of Dynamical Invariance}\label{subsec:shodynam}

Fix $n=2$.  It is convenient to use the definitions $s_j=\frac{1}{2}(p_j^2+q_j^2)$ and $\pp{}{\theta_j}=p_j\pp{}{q_j}-q_j\pp{}{p_j}$ everywhere on $M$, for $j=1,2$.  We will point out when a statement only holds on the neighborhood $U$.

Let $k\in\R$ be a positive constant, and consider the two functions $H_1,H_2:M\rightarrow\R$ given by $$H_1=s_1+s_2-k,\ \ H_2=(s_1+s_2-k)(s_1+2s_2+1).$$
The function $H_1$ is just the energy function for the harmonic oscillator, shifted by $k$:  its quantized energy levels are $$E_N=\hbar N-k,\ \ N\in\Z.$$  Notice that $H_1^{-1}(0)=H_2^{-1}(0)$.  Let this shared level set be $S$.  The energy $E=0$ is a quantized energy for the system $(M,\omega,H_1)$ if and only if $\frac{k}{\hbar}\in\Z$.

The two Hamiltonian vector fields are
\begin{align*}
\xi_{H_1}&=-\pp{}{\theta_1}-\pp{}{\theta_2},\\ 
\xi_{H_2}&=(s_1+2s_2+1)\xi_{H_1}-H_1\lb{\pp{}{\theta_1}+2\pp{}{\theta_2}}\\
&=-(2s_1+3s_2-k+1)\pp{}{\theta_1}-(3s_1+4s_2-2k+1)\pp{}{\theta_2}.
\end{align*}
Since $H_1=0$ on $S$, we see that $\xi_{H_2}=(s_1+2s_2+1)\xi_{H_1}$ everywhere on $S$.  Thus the vector fields are parallel on $S$, as expected, and so they share the same orbits in $S$.  However, $\xi_{H_1}$ and $\xi_{H_2}$ are not parallel away from $S$.  Let $\phi^t$ be the flow of $\xi_{H_1}$, and let $\rho^t$ be the flow of $\xi_{H_2}$. 

Consider the initial point $m_0\in S\cap U$, where $m_0=(s_{01},s_{02},0,0)$ with $s_{01}+s_{02}=k$ and $s_{01},s_{02}\neq 0$.  The orbit of both $\phi^t$ and $\rho^t$ through $m_0$ is $$\sC=\cb{(s_{01},s_{02},\tau,\tau):\tau\in\R/2\pi\Z}.$$  From Section \ref{subsec:elevels}, we know that $$\txi_{H_1}(m,I)=(\xi_{H_1}(m),0),\ \ \forall m\in\sC,$$ and therefore $$\tp^t(m_0,I)=(\phi^t(m_0),I).$$  By the same calculation,  $\txi_{H_2}(m,I)=(\xi_{H_2}(m),\left.\dd{}{t}\right|_{t=0}\rho^t_*|_m)$ for $m\in\sC$.  Applying Equation (\ref{eq:dvfld}) to the components of $\xi_{H_2}$ yields  $$\left.\dd{}{t}\right|_{t=0}\rho^t_*|_{m}=\lb{\begin{array}{cccc}0&0&0&0\\0&0&0&0\\-2&-3&0&0\\-3&-4&0&0\end{array}},$$ which we interpret as an element of $\mathfrak{sp}(V)$.  Let this matrix be denoted by $\kappa$.  Then $$\txi_{H_2}(m,I)=(\xi_{H_2}(m),\kappa),$$ and since the Lie algebra component is constant over $\sC$, we obtain the integral curve on $\Sp(M,\omega)$ through $(m_0,I)$ by exponentiating: $$\tr^t(m_0,I)=(\rho^t(m_0),\exp(t\kappa)).$$  A calculation establishes that $$\exp(t\kappa)=\lb{\begin{array}{cccc}1&0&0&0\\0&1&0&0\\-2t&-3t&1&0\\-3t&-4t&0&1\end{array}},$$ which is clearly not periodic.  Thus $\tr^t$ has no closed orbits on $\Sp(M,\omega)$ over $S\cap U$.  

Now let us transfer to $\Sp(TS/TS^\perp)$.  If we apply the definitions and identifications laid out at the beginning of Section \ref{subsec:elevels}, now setting $n=2$, then we find $W^\perp=\mbox{span}\cb{\hy_1+\hy_2}$, $W/W^\perp=\mbox{span}\cb{\ls{\hx_1-\hx_2},\ls{\hy_1-\hy_2}}$, and the identification of $\Sp(M,\omega)|_{\sC}$ with $\sC\times\Sp(V)$ induces an identification of $\Sp(TS/TS^\perp)|_{\sC}$ with $\sC\times\Sp(W/W^\perp)$.  Notice that
\begin{align*}
\exp(t\kappa)(\hx_1-\hx_2)&=\hx_1-\hx_2+t(\hy_1+\hy_2),\\
\exp(t\kappa)(\hy_1-\hy_2)&=\hy_1-\hy_2.
\end{align*}
Therefore the path through $\Sp(W/W^\perp)$ induced by $\exp(t\kappa)$ is $$\nu(\exp(t\kappa))=\lb{\begin{array}{cc}1&0\\0&1\end{array}}.$$  Thus, on $\Sp(TS/TS^\perp)|_{\sC}$, $$\tr^t(m_0,I)=\lb{\rho^t(m_0),I},$$ which coincides with $\tp^t(m_0,I)$.

The above calculation omits certain cases:  namely, if the starting point $m_0$ has $s_{01}=0$ or $s_{02}=0$.  We can no longer eliminate such cases by performing a rotation, because $H_2$ is not symmetric with respect to $s_1$ and $s_2$.  If $m_0\notin U$, then we have to modify our approach.  For example, if $s_{01}=0$, then we retain Cartesian coordinates for the $p_1q_1$-plane and convert to symplectic polar on the $p_2q_2$-plane.  The calculation is more complicated, but the result is similar:  over $\sC$, we find that $\txi_{H_2}=(\xi_{H_2},\kappa)$ for some constant value of $\kappa\in\mathfrak{sp}(V)$.  The path $\tr^t(m_0,I)=(\rho^t(m_0),\exp(t\kappa))$ does not close in $\Sp(M,\omega)$, but the induced path in $\Sp(TS/TS^\perp)$ coincides with $\tp^t(m_0,I)$.  The same pattern holds if we take $s_{02}=0$.

Thus, if the quantized energy condition were stated in terms of the holonomy of $\gamma^S$ over closed orbits in $\Sp(M,\omega;S)$, as it was in \cite{rob1}, then the value $E=0$ would satisfy the condition vacuously for the system $(M,\omega,H_2)$, regardless of the value of $k$.  It is only by descending to $\Sp(TS/TS^\perp)$ that we recover the quantization condition $\frac{k}{\hbar}\in\Z$.  Hence our definition of a quantized energy level is dynamically invariant, while that in \cite{rob1} is not.

\section{Comparison with Kostant-Souriau Quantization}\label{sec:ks}

\subsection{The quantized energy condition and its properties}

The quantized energy condition in Definition \ref{def:quantE} can be easily adapted to Kostant-Souriau prequantization.  Indeed, the Kostant-Souriau case is simpler, since it does not involve the symplectic frame bundle.  First, recall the definition of a prequantization circle bundle for a symplectic manifold $(M,\omega)$.

\begin{definition}
A \textbf{prequantization circle bundle} for $(M,\omega)$ is a principal circle bundle $Y\maps{\Pi}M$, together with a $\mathfrak{u}(1)$-valued connection one-form $\gamma$ on $Y$ such that $d\gamma=\frac{1}{i\hbar}\Pi^*\omega$.
\end{definition}

Assume that $(M,\omega)$ admits a prequantization circle bundle $(Y,\gamma)$, and let $H:M\rightarrow\R$ be a smooth function.  We now mimic the constructions in Sections \ref{subsec:hamvt} and \ref{subsec:mpcred}, using $Y$ in place of $P$.  Denote the Hamiltonian vector field corresponding to $H$ by $\xi_H$ as before, and let $\txi_H$ be the lift of $\xi_H$ to $Y$ that is horizontal with respect to $\gamma$.  Fix $E$, a regular value of $H$, and let $S=H^{-1}(E)\subset M$.  Let $Y^S$ be the restriction of $Y$ to $S$, and let $\gamma^S$ be the pullback of $\gamma$ to $Y^S$.  
$$\xymatrix{(Y,\gamma) \ar[d] & \ar[l]_{\mbox{\footnotesize incl.}} (Y^S,\gamma^S) \ar[d] \\
(M,\omega) & \ar[l]_{\mbox{\footnotesize incl.}} S
}$$
Since there is no group homomorphism with which to construct an associated bundle, the construction stops here and we give the definition of a quantized energy level using $(Y^S,\gamma^S)$.


\begin{definition}
If the connection one-form $\gamma^S$ has trivial holonomy over all closed orbits of the Hamiltonian vector field $\xi_H$ on $S$, then $E$ is a \textbf{Kostant-Souriau (KS) quantized energy level} for the system $(M,\omega,H)$.
\end{definition}

It is straightforward to establish that this definition has a dynamical invariance property.  Suppose $H_1,H_2:M\rightarrow\R$ are smooth functions such that $H_1^{-1}(E_1)=H_2^{-1}(E_2)=S$ for regular values $E_1$ and $E_2$.  We argued in Section \ref{subsec:stage2} that $\xi_{H_1}$ and $\xi_{H_2}$ are parallel on $S$, which implies that they have the same orbits.  Therefore $\gamma^S$ has trivial holonomy over the orbits of one if and only if it has trivial holonomy over the orbits of the other.  The KS version of the dynamical invariance theorem is immediate.

\begin{theorem}
If $H_1,H_2:M\rightarrow\R$ are smooth functions such that $H_1^{-1}(E_1)=H_2^{-1}(E_2)$ for regular values $E_j$ of $H_j$, $j=1,2$, then $E_1$ is a KS quantized energy level for $(M,\omega,H_1)$ if and only if $E_2$ is a KS quantized energy level for $(M,\omega,H_2)$.
\end{theorem}

In Section \ref{subsec:mpcred}, we examined the case in which the symplectic reduction of $(M,\omega)$ at $E$ is a manifold.  Theorem \ref{thm:rob}, due to Robinson \cite{rob1}, gives the conditions under which the quantization condition is sufficient to imply that the symplectic reduction admits a metaplectic-c quantization.  A similar theorem can be given in the context of prequantization circle bundles.

First, we state a general result, which was used in \cite{rob1} to prove Theorem \ref{thm:rob}.  Let $S$ be an arbitrary manifold, and suppose that $(Z,\delta)$ is a principal circle bundle with connection one-form over $S$.  Let the curvature of $\delta$ be $\varpi$.  Suppose that $F$ is a foliation of $S$ whose leaf space $S_F$ is a smooth manifold.  Denote the leaf projection map by $S\maps{\pi}S_F$.  If $\delta$ has trivial holonomy over all of the leaves of $F$, then $Z$ can be factored to produce a well-defined circle bundle $Z_F\rightarrow S_F$.  Further, $\delta$ descends to a connection one-form $\delta_F$ on $Z_F$, and the curvature $\varpi_F$ of $\delta_F$ satisfies $\pi^*\varpi_F=\varpi$.

Now apply this result to the circle bundle $(Y^S,\gamma^S)\rightarrow S$, where the foliation is given by the orbits of the vector field $\xi_H$ on the level set $S$, and the symplectic reduction $(M_E,\omega_E)$ is its leaf space.  

\begin{theorem}
Suppose that the symplectic reduction $(M_E,\omega_E)$ for $(M,\omega)$ at $E$ is a manifold.  If $\gamma^S$ has trivial holonomy over all closed orbits of $\xi_H$ on $S$, then the quotient of $(Y^S,\gamma^S)$ by the orbits of $\txi_H$ is a prequantization circle bundle for $(M_E,\omega_E)$.
\end{theorem}

Thus the Kostant-Souriau version of the quantized energy condition is sufficient to ensure that the symplectic reduction admits a prequantization circle bundle, whenever the symplectic reduction is a manifold.  This is a slight improvement over the metaplectic-c result, since it does not depend on a quotient of the symplectic frame bundle being well defined.

\subsection{Lack of half-shift in the harmonic oscillator}

In this section, we will determine the quantized energy levels of the $n$-dimensional harmonic oscillator, using a prequantization circle bundle and the KS version of the quantization condition.  The calculation will be significantly simpler than that in Section \ref{sec:sho}, but it will yield the wrong answer.

Let $M=\R^{2n}$, with Cartesian coordinates $(p_1,\ldots,p_n,q_1,\ldots,q_n)$ and symplectic form $\omega=\sum_{j=1}^ndp_j\wedge dq_j$.  The energy function and corresponding Hamiltonian vector field for the harmonic oscillator are $$H=\frac{1}{2}\sum_{j=1}^n(p_j^2+q_j^2),\ \ \xi_H=\sum_{j=1}^n\lb{q_j\pp{}{p_j}-p_j\pp{}{q_j}}.$$  Let the flow of $\xi_H$ on $M$ be $\phi^t$.  We know from Section \ref{subsec:shoinit} that all of the orbits of $\xi_H$ are circles, and that  $\phi^{t+2\pi}(m)=\phi^t(m)$ for all $m\in M$.

Let $Y=M\times U(1)$, with projection map $Y\maps{\Pi}M$.  We define  $$\beta=\frac{1}{2}\sum_{j=1}^n(p_jdq_j-q_jdp_j)$$ on $M$, and let $\gamma=\frac{1}{i\hbar}\Pi^*\beta+\vartheta_0$ on $Y$, where $\vartheta_0$ is the trivial connection on the product bundle $M\times U(1)$.  Then $(Y,\gamma)$ is a prequantization circle bundle for $(M,\omega)$, and it is unique up to isomorphism.

Let $E>0$ be arbitrary, and let $S=H^{-1}(E)$.  At any point $s\in S$, we calculate that $\xi_H\contr\beta=-E$.  Therefore the lifted vector field $\txi_H$ at the point $(s,I)\in Y$ is $$\txi_H(s,I)=\lb{\xi_H(s),\frac{E}{i\hbar}},$$ where we identify the tangent space $T_{(s,I)}Y$ with $T_sM\times\mathfrak{u}(1)$.  Note that the $\mathfrak{u}(1)$ component is constant over all of $S$, so in particular it is constant over an orbit.  Let the flow of $\txi_H$ be $\tp^t$.  By exponentiation, we find that $$\tp^t(s,I)=\lb{\phi^t(s),e^{Et/i\hbar}}.$$  From this, it follows that the holonomy of $\gamma^S$ over the orbit is trivial if and only if $E=N\hbar$ for some $N\in\Z$.  Thus the KS quantized energy levels are inconsistent with the predictions of quantum mechanics when the dimension $n$ is odd.

This shortcoming in the Kostant-Souriau prequantization of the harmonic oscillator is well known, and the standard solution is to proceed from prequantization to quantization while introducing the half-form correction.  We briefly sketch the process; see, for example, \cite{gs1,sw1,sn1} for more details.  

Suppose a symplectic manifold $(M,\omega)$ admits a prequantization circle bundle $(Y,\gamma)$.  Let $(L,\nabla)$ be the complex line bundle with connection associated to $(Y,\gamma)$.  To quantize $(M,\omega)$, we require two more objects:  a metaplectic structure for $(M,\omega)$, and a choice of polarization $F\subset TM^\C$.  The metaplectic structure induces a metalinear structure on the polarization, from which we can construct $\wedge^{1/2}F\rightarrow M$, the complex line bundle of half-forms.  The quantization of $(M,\omega)$ is a representation of a certain subalgebra of $C^\infty(M)$ as operators on those sections of $L\otimes\wedge^{1/2}F$ that are flat along the leaves of $F$.  

In this context, the quantized energy levels for the system $(M,\omega,H)$ can be taken to be the eigenvalues of the operator corresponding to the energy function $H$.  When the quantization recipe is applied to the harmonic oscillator, the presence of the half-form bundle adds the $\frac{n}{2}$ shift to the energy eigenvalues.  However, this correction comes at the cost of introducing a choice of polarization, and the quantized energy definition can no longer be evaluated over a single level set of $H$.

By comparison, when we use the metaplectic-c formulation of a quantized energy level, we find that the correct harmonic oscillator energies are encoded in the geometry of the level sets of the energy function.  The result is dynamically invariant, independent of polarization, and consistent with physical prediction.  This example illustrates the benefits of metaplectic-c quantization and our quantized energy definition.

\subsection*{Acknowledgements}

The author thanks her advisor, Yael Karshon, for many helpful discussions.  This work was supported in part by an NSERC scholarship.

\end{document}